\documentclass[12pt]{article}

\usepackage{amssymb,amsmath}
\usepackage[margin=0.75in]{geometry}
\usepackage{mathrsfs}
\usepackage[all]{xy}
\usepackage[usenames,dvipsnames,svgnames,table]{xcolor}
\usepackage{tikz}
\usepackage{gensymb}
\usepackage[shortlabels]{enumitem}
\usepackage{float}
\usepackage{subcaption}
\usepackage{hyperref}
\usepackage[numbers]{natbib}
\usepackage{csquotes}

\newtheorem{theorem}{Theorem}[subsection]

\newtheorem{definition}[theorem]{Definition}
\newtheorem{example}[theorem]{Example}

\newtheorem{lemma}[theorem]{Lemma}

\numberwithin{equation}{section}

\def\qed{\hfill {\hbox{${\vcenter{\vbox{             
   \hrule height 0.4pt\hbox{\vrule width 0.4pt height 6pt
   \kern5pt\vrule width 0.4pt}\hrule height 0.4pt}}}$}}}

\newenvironment{proof}[1][Proof]{\smallskip\noindent{\bf #1.}\quad}
{\qed\par\medskip}

\begin{document}



\begin{center}
\Large Into the Woods: A graph theoretic proof of the Dodgson/Muir identity\\
Melanie Fraser
\end{center}

\normalsize

\textbf{Abstract.} The Dodgson/Muir Identity is an identity on determinants of matrix minors that generalizes the Dodgson Identity. Using the matrix tree theorem, we create an equivalent forest identity on ordered sets of $k$ forests. We then prove the generalized Forest Identity, and by extension the Dodgson/Muir Identity, using an edge-swapping involution. This algorithm is a generalization of the Red Hot Potato algorithm, developed by the author in 2021 to prove the Dodgson Identity.

\section{Introduction}

This paper presents a graph theoretic interpretation of the Dodgson/Muir Identity. The Dodgson/Muir Identity is a generalization of Dodgson Identity, also known as the Lewis Carroll Identity, which has been used to develop Dodgson polynomials to study Feynman graphs and perterbative quantum field theory \cite{brown}, \cite{golz}, \cite{schnetz}. The Dodgson/Muir Identity is a determinantal identity which was given a combinatorial proof by Berliner and Brualdi \cite{brualdi}. Their proof is a generalization of Zeilberger's proof of Dodgson's Identity \cite{zeilberger}. The author gave a graph theoretic interpretation of Zeilberger's proof by interpreting Dodgson's Identity as a Forest Identity, and then proving the Forest Identity directly using an edge swapping argument \cite{fraser}. This paper will similarly interpret the Dodgson/Muir Identity as the Generalized Forest Identity, and then generalize the edge swapping argument in \cite{fraser} by extending the algorithm in much the same way as Berliner and Brualdi extended Zeilberger's combinatorial proof of Dodgson's Identity.

\begin{definition}
Let $U$ and $W$ be sets of nodes of the same size, and let $M[U,W]$ be the matrix $M$ with only the rows corresponding to the elements in $U$ and the columns corresponding to the elements in $W$ (with the rest of the rows and columns in $M$ removed).
\end{definition}

\begin{theorem}
\textbf{Dodgson/Muir Identity.} Let $M$ be a square $n\times n$ matrix, and let $k$ be any integer $1\leq k\leq n$. Then
\begin{equation}
\begin{split}
\det(M)\cdot\det&(M[\{k+1,\dots,n\},\{k+1,\dots,n\}])^{k-1}=\\
&\sum_{\sigma\in S_k}(-1)^{\iota(\sigma)}\prod_{i=1}^k\det(M[\{i,k+1,\dots,n\},\{\sigma(i),k+1,\dots,n\}]).
\end{split}
\end{equation}
\end{theorem}

Notice that when $k=2$, we get Dodgson's Identity. When $k=n$, we get the permutation definition of a determinant (since $M[\emptyset,\emptyset]$ is an empty matrix with determinant $1$).

\begin{example}
Let $M=\bordermatrix{&1'&2'&3'&4'\cr 1&3&7&0&0\cr 2&8&1&0&0\cr 3&0&0&4&0\cr 4&0&0&0&2}$, and let $k=3$. Then $\det(M)=-424$ and $\det(M[4,4])=2$, so the left hand side of the equation gives us $(-424)\cdot(2)^2=-1696$. For the right hand side, notice that only two of the permutations in $S_3$ give us a non-zero product, namely $\sigma_1=123$ and $\sigma_2=213$. For $\sigma_1$, we have that $\det(M[\{1,4\},\{1,4\}])=6$, $\det(M[\{2,4\},\{2,4\}])=2$, and $\det(M[\{3,4\},\{3,4\}])=8$. Since $\sigma_1$ has no inversions, its entry in the sum is $96$. For $\sigma_2$, we have that $\det(M[\{1,4\},\{2,4\}])=14$, $\det(M[\{2,4\},\{1,4\}])=16$, and $\det(M[\{3,4\},\{3,4\}])=8$. Since $\sigma_2$ has one inversion, it contributes negatively to the sum, giving us $-1792$. Adding the two terms of the right hand side together gives us that the right hand side also equals $-1696$.
\end{example}

We will interpret the Dodgson/Muir Identity in terms of sets of directed rooted forests. In our notion of a tree, each node has either one out-edge, or no out edges (it can have any number of in-edges). A root is a node in a tree with no out-edges such that every other node in the tree is connected to the root by a directed path. Thus in a directed rooted forest, every node has one out-edge except for the roots, which are exactly the nodes with no out-edges. We will call a forest made up of $k$ trees a $k$-forest.

We call a path from node $i$ to node $j$ a \textbf{meta-edge} $i\to j$. In a set of forests, if we have a meta-edge $i\to j$ in one forest, $j\to k$ in another, and so on before eventually having a forest with meta-edge $m\to i$, we call this a \textbf{meta-cycle}. For instance, in the following example, the edges $3\to 4$ and $4\to 2$ together form the meta-edge $3\to 2$; and the edge $1\to 3$ forms the meta-edge $1\to 3$; and the edges $2\to 5$ and $5\to 1$ together form the meta-edge $2\to 1$. Together these three meta-edges form a meta-cycle $3\to 2\to 1 \to 3$.

\begin{center}
\begin{tikzpicture}[baseline=0ex]
\node (0) at(0,0) {0};
\node (1) at(1,.5) {1};
\node (2) at(2,0) {2};
\node (3) at(2,-1) {3};
\node (4) at(1,-1.5) {4};
\node (5) at(0,-1) {5};
\draw[thick, ->, red, dashed] (3) edge (4);
\draw[thick, ->, red, dashed] (4) edge (2);
\draw[thick, ->] (5) edge (1);

\node (0) at(4,0) {0};
\node (1) at(5,.5) {1};
\node (2) at(6,0) {2};
\node (3) at(6,-1) {3};
\node (4) at(5,-1.5) {4};
\node (5) at(4,-1) {5};
\draw[thick, ->, red, dashed] (1) edge (3);
\draw[thick, ->] (4) edge (5);
\draw[thick, ->] (5) edge (2);

\node (0) at(8,0) {0};
\node (1) at(9,.5) {1};
\node (2) at(10,0) {2};
\node (3) at(10,-1) {3};
\node (4) at(9,-1.5) {4};
\node (5) at(8,-1) {5};
\draw[thick, ->, red, dashed] (2) edge (5);
\draw[thick, ->] (4) edge (2);
\draw[thick, ->, red, dashed] (5) edge (1);
\end{tikzpicture}
\end{center}

To give the Generalized Forest Identity, we will need to introduce some notation. Let $[k]=\{0,1,\dots,k\}$. Given weights on each directed edge of a forest, the forest weight $a_F$ is the product of each edge weight. Given a set $U$ of nodes, we will let $R_U$ be the set of forests with roots at exactly the nodes in $U$. Let $R_U^{i\to j}$ be the set of forests with roots in $U$ that also have a meta-edge from node $i$ to node $j$. For example, $R_{[k]\setminus\{2\}}^{2\to 4}$ is the set of forests with roots at $0,1,3,4,\dots,k$ that have a meta-edge $2\to 4$.
Finally, for a permutation $\sigma\in S_k$ for some $k$, let $R^{\sigma}=R_{[k]\setminus\{1\}}^{1\to\sigma(1)}\times R_{[k]\setminus\{2\}}^{2\to\sigma(2)}\times\cdots\times R_{[k]\setminus\{k\}}^{k\to\sigma(k)}$. We will let $R^{NF}$ stand for non-forbidden forests, which are the $k$-forests that are allowed as part of the Generalized Forest Identity, as defined below.

\begin{theorem}\label{genforestidentity}
\textbf{Generalized Forest Identity.} For any $k$, let $\sigma_{\text{id}}$ be the identity permutation in $S_k$, and let $R^{NF}$ be the set $R^{\sigma_{\text{id}}}\setminus \bigcup_{\sigma\in S_k, \sigma\neq \sigma_{\text{id}}}R^{\sigma}$. Then
\begin{equation}
\sum_{(F_1,F_2,\dots,F_k)\in R_0\times (R_{[k]'})^{k-1}}\prod_{i=1}^k a_{F_i}=\sum_{(F_1,F_2,\dots,F_k)\in R^{NF}}\prod_{i=1}^k a_{F_i}.
\end{equation}
\end{theorem}

The left hand side of this identity represents ordered sets of $k$ forests, one a tree rooted at zero and the remaining $k-1$ all $k+1$-forests rooted at $0,1,\dots,k$. The right hand side represents ordered sets of $k$ forests, the $i$th of which is a $k$-forest rooted at $0,1,\dots,i-1,i+1,\dots k$ such that there are no meta-cycles involving the first $k$ nodes. Since the $i$th forest is the only forest in the set with an edge out of node $i$, this means there are no meta-cycles such that the $i$th forest has meta-edge $i\to a_1$, the $a_1$-th forest has meta-edge $a_1\to a_2$, and so on until the $a_m$-th forest has meta-edge $a_m\to i$ where $i,a_j$ for all $1\leq j\leq m$ are all in the set $[k]$. A meta-cycle of this description is called a \textbf{forbidden meta-cycle}, and a set of $k$ forests with a forbidden meta-cycle is called a set of \textbf{forbidden forests}.

\begin{example}
Below is the case $n=k=3$:

We will begin with finding $R^{\sigma}$ for each $\sigma\in S_3$ not the identity.

$R^{132}$:
$\left(\begin{tikzpicture}[baseline=0ex]
\node (0) at(0,0) {0};
\node (1) at(1,1) {1};
\node (2) at(2,0) {2};
\node (3) at(1,-1) {3};
\draw[thick, ->] (1) edge (0);
\end{tikzpicture}
+
\begin{tikzpicture}[baseline=0ex]
\node (0) at(0,0) {0};
\node (1) at(1,1) {1};
\node (2) at(2,0) {2};
\node (3) at(1,-1) {3};
\draw[thick, ->] (1) edge (2);
\end{tikzpicture}
+
\begin{tikzpicture}[baseline=0ex]
\node (0) at(0,0) {0};
\node (1) at(1,1) {1};
\node (2) at(2,0) {2};
\node (3) at(1,-1) {3};
\draw[thick, ->] (1) edge (3);
\end{tikzpicture}\right)
\begin{tikzpicture}[baseline=0ex]
\draw[fill] (0,0) circle [radius=.05];
\end{tikzpicture}$

\hspace{15mm}$\left(\begin{tikzpicture}[baseline=0ex]
\node (0) at(0,0) {0};
\node (1) at(1,1) {1};
\node (2) at(2,0) {2};
\node (3) at(1,-1) {3};
\draw[thick, ->] (2) edge (3);
\end{tikzpicture}\right)\begin{tikzpicture}[baseline=0ex]
\draw[fill] (0,0) circle [radius=.05];
\end{tikzpicture}$
$\left(\begin{tikzpicture}[baseline=0ex]
\node (0) at(0,0) {0};
\node (1) at(1,1) {1};
\node (2) at(2,0) {2};
\node (3) at(1,-1) {3};
\draw[thick, ->] (3) edge (2);
\end{tikzpicture}\right)$

$R^{213}$:
$\left(\begin{tikzpicture}[baseline=0ex]
\node (0) at(0,0) {0};
\node (1) at(1,1) {1};
\node (2) at(2,0) {2};
\node (3) at(1,-1) {3};
\draw[thick, ->] (1) edge (2);
\end{tikzpicture}\right)\begin{tikzpicture}[baseline=0ex]
\draw[fill] (0,0) circle [radius=.05];
\end{tikzpicture}$
$\left(\begin{tikzpicture}[baseline=0ex]
\node (0) at(0,0) {0};
\node (1) at(1,1) {1};
\node (2) at(2,0) {2};
\node (3) at(1,-1) {3};
\draw[thick, ->] (2) edge (1);
\end{tikzpicture}\right)\begin{tikzpicture}[baseline=0ex]
\draw[fill] (0,0) circle [radius=.05];
\end{tikzpicture}$

\hspace{15mm}
$\left(\begin{tikzpicture}[baseline=0ex]
\node (0) at(0,0) {0};
\node (1) at(1,1) {1};
\node (2) at(2,0) {2};
\node (3) at(1,-1) {3};
\draw[thick, ->] (3) edge (0);
\end{tikzpicture}
+
\begin{tikzpicture}[baseline=0ex]
\node (0) at(0,0) {0};
\node (1) at(1,1) {1};
\node (2) at(2,0) {2};
\node (3) at(1,-1) {3};
\draw[thick, ->] (3) edge (1);
\end{tikzpicture}
+
\begin{tikzpicture}[baseline=0ex]
\node (0) at(0,0) {0};
\node (1) at(1,1) {1};
\node (2) at(2,0) {2};
\node (3) at(1,-1) {3};
\draw[thick, ->] (3) edge (2);
\end{tikzpicture}\right)$

$R^{231}$:
$\left(\begin{tikzpicture}[baseline=0ex]
\node (0) at(0,0) {0};
\node (1) at(1,1) {1};
\node (2) at(2,0) {2};
\node (3) at(1,-1) {3};
\draw[thick, ->] (1) edge (2);
\end{tikzpicture}\right)\begin{tikzpicture}[baseline=0ex]
\draw[fill] (0,0) circle [radius=.05];
\end{tikzpicture}$
$\left(\begin{tikzpicture}[baseline=0ex]
\node (0) at(0,0) {0};
\node (1) at(1,1) {1};
\node (2) at(2,0) {2};
\node (3) at(1,-1) {3};
\draw[thick, ->] (2) edge (3);
\end{tikzpicture}\right)\begin{tikzpicture}[baseline=0ex]
\draw[fill] (0,0) circle [radius=.05];
\end{tikzpicture}$
$\left(\begin{tikzpicture}[baseline=0ex]
\node (0) at(0,0) {0};
\node (1) at(1,1) {1};
\node (2) at(2,0) {2};
\node (3) at(1,-1) {3};
\draw[thick, ->] (3) edge (1);
\end{tikzpicture}\right)$

$R^{312}$:
$\left(\begin{tikzpicture}[baseline=0ex]
\node (0) at(0,0) {0};
\node (1) at(1,1) {1};
\node (2) at(2,0) {2};
\node (3) at(1,-1) {3};
\draw[thick, ->] (1) edge (3);
\end{tikzpicture}\right)\begin{tikzpicture}[baseline=0ex]
\draw[fill] (0,0) circle [radius=.05];
\end{tikzpicture}$
$\left(\begin{tikzpicture}[baseline=0ex]
\node (0) at(0,0) {0};
\node (1) at(1,1) {1};
\node (2) at(2,0) {2};
\node (3) at(1,-1) {3};
\draw[thick, ->] (2) edge (1);
\end{tikzpicture}\right)\begin{tikzpicture}[baseline=0ex]
\draw[fill] (0,0) circle [radius=.05];
\end{tikzpicture}$
$\left(\begin{tikzpicture}[baseline=0ex]
\node (0) at(0,0) {0};
\node (1) at(1,1) {1};
\node (2) at(2,0) {2};
\node (3) at(1,-1) {3};
\draw[thick, ->] (3) edge (2);
\end{tikzpicture}\right)$

$R^{321}$:
$\left(\begin{tikzpicture}[baseline=0ex]
\node (0) at(0,0) {0};
\node (1) at(1,1) {1};
\node (2) at(2,0) {2};
\node (3) at(1,-1) {3};
\draw[thick, ->] (1) edge (3);
\end{tikzpicture}\right)\begin{tikzpicture}[baseline=0ex]
\draw[fill] (0,0) circle [radius=.05];
\end{tikzpicture}$
$\left(\begin{tikzpicture}[baseline=0ex]
\node (0) at(0,0) {0};
\node (1) at(1,1) {1};
\node (2) at(2,0) {2};
\node (3) at(1,-1) {3};
\draw[thick, ->] (2) edge (0);
\end{tikzpicture}
+
\begin{tikzpicture}[baseline=0ex]
\node (0) at(0,0) {0};
\node (1) at(1,1) {1};
\node (2) at(2,0) {2};
\node (3) at(1,-1) {3};
\draw[thick, ->] (2) edge (1);
\end{tikzpicture}
+
\begin{tikzpicture}[baseline=0ex]
\node (0) at(0,0) {0};
\node (1) at(1,1) {1};
\node (2) at(2,0) {2};
\node (3) at(1,-1) {3};
\draw[thick, ->] (2) edge (3);
\end{tikzpicture}\right)\begin{tikzpicture}[baseline=0ex]
\draw[fill] (0,0) circle [radius=.05];
\end{tikzpicture}$

\hspace{15mm}
$\left(\begin{tikzpicture}[baseline=0ex]
\node (0) at(0,0) {0};
\node (1) at(1,1) {1};
\node (2) at(2,0) {2};
\node (3) at(1,-1) {3};
\draw[thick, ->] (3) edge (1);
\end{tikzpicture}\right)$

Notice that the single elements in each product together form forbidden meta-cycles, which is why all of the above possibilities are forbidden. There are a total of $11$ forbidden forests listed above. To get the right hand side, we remove these from the total set of $27$ from the identity, listed below, leaving $16$ elements:

$
\left(\begin{tikzpicture}[baseline=0ex]
\node (0) at(0,0) {0};
\node (1) at(1,1) {1};
\node (2) at(2,0) {2};
\node (3) at(1,-1) {3};
\draw[thick, ->] (1) edge (0);
\end{tikzpicture}
+
\begin{tikzpicture}[baseline=0ex]
\node (0) at(0,0) {0};
\node (1) at(1,1) {1};
\node (2) at(2,0) {2};
\node (3) at(1,-1) {3};
\draw[thick, ->] (1) edge (2);
\end{tikzpicture}
+
\begin{tikzpicture}[baseline=0ex]
\node (0) at(0,0) {0};
\node (1) at(1,1) {1};
\node (2) at(2,0) {2};
\node (3) at(1,-1) {3};
\draw[thick, ->] (1) edge (3);
\end{tikzpicture}\right)\begin{tikzpicture}[baseline=0ex]
\draw[fill] (0,0) circle [radius=.05];
\end{tikzpicture}$

$\left(\begin{tikzpicture}[baseline=0ex]
\node (0) at(0,0) {0};
\node (1) at(1,1) {1};
\node (2) at(2,0) {2};
\node (3) at(1,-1) {3};
\draw[thick, ->] (2) edge (0);
\end{tikzpicture}
+
\begin{tikzpicture}[baseline=0ex]
\node (0) at(0,0) {0};
\node (1) at(1,1) {1};
\node (2) at(2,0) {2};
\node (3) at(1,-1) {3};
\draw[thick, ->] (2) edge (1);
\end{tikzpicture}
+
\begin{tikzpicture}[baseline=0ex]
\node (0) at(0,0) {0};
\node (1) at(1,1) {1};
\node (2) at(2,0) {2};
\node (3) at(1,-1) {3};
\draw[thick, ->] (2) edge (3);
\end{tikzpicture}\right)\begin{tikzpicture}[baseline=0ex]
\draw[fill] (0,0) circle [radius=.05];
\end{tikzpicture}$

$\left(\begin{tikzpicture}[baseline=0ex]
\node (0) at(0,0) {0};
\node (1) at(1,1) {1};
\node (2) at(2,0) {2};
\node (3) at(1,-1) {3};
\draw[thick, ->] (3) edge (0);
\end{tikzpicture}
+
\begin{tikzpicture}[baseline=0ex]
\node (0) at(0,0) {0};
\node (1) at(1,1) {1};
\node (2) at(2,0) {2};
\node (3) at(1,-1) {3};
\draw[thick, ->] (3) edge (1);
\end{tikzpicture}
+
\begin{tikzpicture}[baseline=0ex]
\node (0) at(0,0) {0};
\node (1) at(1,1) {1};
\node (2) at(2,0) {2};
\node (3) at(1,-1) {3};
\draw[thick, ->] (3) edge (2);
\end{tikzpicture}\right)$

For the left hand side, we get any of the possible $16$ trees rooted at zero as the first element, followed by two copies of the $4$-forest with nodes $0,1,2,3$ and no edges. This gives us a total of $16$ elements in the left hand side. We see that taking an element from the right hand side and combining all of the edges into one graph while leaving the other two empty will give us an element from the left hand side.

\end{example}

\section[Connecting Identities]{Matrix Tree Theorem connection between Generalized Forest and Dodgson/Muir Identities}

We can use the All Minors Matrix Tree Theorem \cite{chaiken} to derive the Generalized Forest Identity from the Dodgson/Muir Identity.

\begin{definition}
Let $a_ij$ be the weight of the edge $i\to j$. Define the \textbf{Laplacian} $A$ by 
\[A_{ij}= \begin{cases} 
      -a_{ij} & i\neq j \\
      \displaystyle\sum_{k\neq i}a_{ik} & i=j
   \end{cases}
\]
\end{definition}

Let $A$ be the Laplacian encoding the complete directed graph on labeled nodes $\{0,1,\dots n\}$. Replace $M$ in the Dodgson/Muir Identity with $A$ with the zeroth row and column removed, that is $A_{0,0}$. Then we obtain the following:

\begin{equation}\label{gendodgson}
\begin{split}
\det(A[\{1,\dots,n\},\{1,\dots&,n\}])\cdot\det(A[\{k+1,\dots,n\},\{k+1,\dots,n\}])^{k-1}=\\
&\sum_{\sigma\in S_k}(-1)^{\iota(\sigma)}\prod_{i=1}^k\det(A[\{i,k+1,\dots,n\},\{\sigma(i),k+1,\dots,n\}]).
\end{split}
\end{equation}

Now we can use the Matrix Tree Theorem \cite{chaiken} to compare Eq. \ref{gendodgson} to the Generalized Forest Identity (Theorem \ref{genforestidentity}). By the All-Minor's Matrix Tree Theorem, the determinant of a minor of the Laplacian corresponds to signed forests rooted at the removed rows in which exactly one node from the rows removed and one node from the columns removed are in each tree. The sign of the forest is the sign it contributes to the determinant of the Laplacian.

The Matrix Tree Theorem tells us that $\det(A[\{1,\dots,n\},\{1,\dots,n\}])$ corresponds to spanning trees rooted at zero and $\det(A[\{k+1,\dots,n\},\{k+1,\dots,n\}])$ corresponds to $k+1$-forests rooted at $0,1,\dots, k$. Thus the left hand side of Eq. \ref{gendodgson} corresponds to the left hand side of the Generalized Forest Identity. The Matrix Tree Theorem also tells us that $\prod_{i=1}^k\det(A[\{i,k+1,\dots,n\},\{\sigma(i),k+1,\dots,n\}])$ will give us sets of $k$ forests, the $i$th of which is a $k$-forest rooted at $0,1,\dots,i-1,i+1,\dots k$. We now need to interpret what the permutation $\sigma$ in the right hand side of Eq. \ref{gendodgson} means.

\begin{lemma}\label{sigma}
In the right hand side of Eq. \ref{gendodgson}, if $\sigma(i)\neq i$, then $i$ is in the tree rooted at $\sigma(i)$. If $\sigma(i)=i$, then $i$ could be in any tree. 
\end{lemma}

\begin{proof}
Let us begin by supposing that $\sigma(i)=i$. Then the indices of rows and columns removed in $\prod_{i=1}^k\det(A[\{i,k+1,\dots,n\},\{\sigma(i),k+1,\dots,n\}])$ match, so there is no restriction on which tree $i$ belongs to. 

Now suppose that $\sigma(i)\neq i$. The Matrix Tree Theorem tells us that if $U$ is the set of indices of removed rows and $W$ is the set of indices of removed columns, then each tree in the resulting forest must contain exactly one element of $U$ and exactly one element of $W$. Then $\sigma(i)$ is a root since every row in $\{0,\dots,k\}$ is removed except for $i$, and $\sigma(i)\neq i$. The columns removed from $A$ are all columns between $0$ and $k$ except for $\sigma(i)$. Thus each of $0$ through $k$ except for $i$ and $\sigma(i)$ are removed as both a row and a column and must therefore be part of their own trees. Then all elements in the rows and columns are automatically paired except for $i$ in the removed columns and $\sigma(i)$ in the removed rows. Thus $i$ must be part of the tree rooted at $\sigma(i)$.
\end{proof}

Notice that as a result of the Lemma \ref{sigma}, $\sigma$ is related to meta-cycles since it dictates which tree node $i$ is a part of. Now we need to determine which permutations in $S_k$ give us negative signs so that we can see which of these sets of forests are kept and which are subtracted away.

We claim that any set of $k$-forests with a forbidden meta-cycle will end up canceling out in the right hand side of Eq. \ref{gendodgson}, and any non-forbidden set of $k$-forests will show up exactly one time. If that is true, then right hand side of Eq. \ref{gendodgson} will correspond to the right hand side of the Generalized Forest Identity (Theorem \ref{genforestidentity}).

We can associate a graph to a permutation array (a matrix with exactly one element from $A$ in each row and column) by putting an edge $i\to j$ in our graph if there is the monomial $-a_{ij}$ in the $i$th row and $j$th column, or if there is the monomial $a_{ij}$ on the diagonal in row $i$. Notice that elements in the permutation arrays can only be off the diagonal (in spaces where their row and column indices do not match) if they form a cycle or meta-cycle in the graph. A cycle can therefore be encoded either on the diagonal or off the diagonal, whereas edges not involved in cycles must be encoded on the diagonal (in spaces with matching row and column indices). If we wish to change one entry from off-diagonal to on-diagonal or vice versa, we must change every entry involved in the cycle or meta-cycle. If a meta-cycle is being encoded on the diagonal, we will color it black, and if it is being encoded off-diagonal, we will color it red.

\begin{definition}
If we change a meta-cycle from being represented on the diagonal to off-diagonal or vice versa, then this is called \textbf{toggling the diagonality} of the associated array. This is the array equivalent of changing the color of a meta-cycle.
\end{definition}

\begin{lemma}\label{toggdiag}
Changing the color of one meta-cycle will change the sign of the set of graphs in the right hand side of Eq. \ref{gendodgson}.
\end{lemma}

\begin{proof}
By Lemma \ref{sigma}, a cycle in a permutation $\sigma$ indicates a red meta-cycle in the set of graphs in the right hand side of Eq. \ref{gendodgson}. The sign of the set of graphs is given both by $\text{sgn}(\sigma)=(-1)^{\iota(\sigma)}$, and also by the sign of the product of determinants $\prod_{i=1}^k\det(A[\{i,k+1,\dots,n\},\{\sigma(i),k+1,\dots,n\}])$ representing the graphs. An alternate way of finding the sign of a permutation is by noting that the parity of inversions in a permutation is equal to the parity of even-length cycles in a permutation. Then each even-length cycle changes the sign of of the permutation, and odd-length cycles do not. Thus changing the color of an even-length meta-cycle adds or removes an even-length cycle to the permutation, and so changes the sign contributed by $\text{sgn}(\sigma)=(-1)^{\iota(\sigma)}$.

The sign of the product of determinants will also contribute to the overall sign of the element in the right hand side of Eq. \ref{gendodgson}. Toggling the diagonality of a particular array changes the sign that the array contributes to the determinant, \cite{sally}. Thus if we toggle the diagonality of an even-length meta-cycle, it toggles the diagonality (and thus changes the sign) of an even number of arrays, so the overall sign contributed to the determinant remains the same. If we toggle the diagonality of an odd-length meta-cycle, it toggles the diagonality of an odd number of arrays, so the overall sign contributed to the determinant will change. Since toggling the diagonality of an array corresponds to changing the color of a meta-cycle, then if we change the color of an odd-length meta-cycle, we will change the sign contributed by the product of determinants.

Combining these two observations, we see that changing the color of one meta-cycle (whether odd or even) will change the sign of the set of graphs in the right hand side of Eq. \ref{gendodgson}.
\end{proof}

\begin{theorem}
Any set of $k$-forests with a meta-cycle will cancel in the sum on the right hand side of Eq. \ref{gendodgson}, and any set of $k$-forests with no meta-cycles will be represented exactly one time.
\end{theorem}

\begin{proof}
Suppose that a set of non-colored $k$-forests has $m$ possible meta-cycles (in a colored graph, each of the $m$ meta-cycles could be colored red if represented off-diagonal or black if represented on-diagonal). Let us consider how many permutations in $S_k$ can result in this set of forests. Every $i$ not involved in one of these $m$ meta-cycles must have $\sigma(i)=i$. By Lemma \ref{sigma}, each of the $m$ meta-cycles can either be represented as a cycle in $\sigma$ (corresponding to the meta-cycle being colored red), or represented as fixed points in $\sigma$ (corresponding to the meta-cycle being colored black). Thus we have $2^m$ permutations representing that set of forests.

We want to show that for all $m\geq 1$, half of the $2^m$ permutations will contribute an overall plus sign to the right hand side of Eq. \ref{gendodgson} and half will contribute a minus sign, so overall the contribution of the forests with meta-cycles will cancel out. By Lemma \ref{toggdiag}, if an even number of meta-cycles get represented as cycles in $\sigma$, then the set of forests will be added to the sum, and if an odd number of meta-cycles get represented as cycles in $\sigma$, then the set of forests will be subtracted from the sum.

We can associate to each permutation a subset of $\{1,2,\dots,m\}$ where the subset tells us which of the meta-cycles should be represented as cycles in $\sigma$. Then the subsets of even size will have an even number of cycles in the permutation and will contribute positively, whereas the subsets of odd size will contribute negatively. But the number of even subsets of $m$ is the same as the number of odd subsets of $m$ for all $m\geq 1$. Thus the number of permutations that add a copy of our set of forests to the right hand side of Eq. \ref{gendodgson} is the same as the number of permutations that subtract a copy. Thus for all $m\geq 1$, a set of forests with $m$ meta-cycles will end up canceling out and will not be counted in the right hand side of Eq. \ref{gendodgson}.

Finally, if there are no meta-cycles in our set of forests, then the set can only be associated with the identity permutation, so it is added exactly one time to the right hand side of Eq. \ref{gendodgson}. 
\end{proof}

Thus the right hand side of Eq. \ref{gendodgson} matches the right hand side of the Generalized Forest Identity. Because Eq. \ref{gendodgson} is a direct consequence of the Dodgson/Muir Identity, the Dodgson/Muir Identity proves the Generalized Forest Identity. Our goal for the rest of the paper is to prove the Generalized Forest Identity directly, and then use it to prove the Dodgson/Muir Identity.

\section{Generalized Red Hot Potato algorithm}

The generalized Red Hot Potato algorithm is a generalization of the Red Hot Potato algorithm used to prove Dodgson's Identity \cite {fraser}. It is is based on a consequence of the Involution Principle \cite{garsia}, \cite{sally}. To start, we restate a few definitions found in \cite{fraser}. 

\begin{definition}
A \textbf{signed set} is a set $S$ that is partitioned into two pieces, $S^+$ and $S^-$, such that $S = S^+\sqcup S^-$. A \textbf{sign-reversing function} on $S$ is a function that sends elements from $S^+$ to $S^-$ and elements from $S^-$ to $S^+$.
\end{definition}

\begin{definition}\label{difference}
Given two signed sets, $S_1$ and $S_2$, the \textbf{difference} $S_1 - S_2$ is the
disjoint union of the two sets such that $(S_1-S_2)^+ = S_1^+ \sqcup S^-_2$
and $(S_1-S_2)^- = S^-_1 \sqcup S^+_2$.
\end{definition}

\begin{theorem}\label{invothm}
Given any sequence of signed sets $S_0, S_1, \dots , S_{m+1}$ where $S_0$ and 
$S_{m+1}$ contain only
positive elements, and sign-reversing involutions $\phi_0,\dots, \phi_k$ where $\phi_i: S_i-S_{i+1} \to S_i-S_{i+1}$, there
is a constructible bijection between $S_0$ and $S_{m+1}$.
\end{theorem}

The Involution Principle begins by applying $\phi_0$ to an element in $S_0$. Subsequent $\phi_i$ are iteratively applied to each output, where the appropriate $\phi_i$ is chosen based on the set that the previous output belongs to.

Our goal will be to find a sequence of signed sets and sign-reversing involutions satisfying the assumptions of Theorem \ref{invothm} such that $S_0$ is the set of $k$ forests matching the left hand side of Theorem \ref{genforestidentity}, and $S_m$ is the set of $k$ forests matching the right hand side of Theorem \ref{genforestidentity}.

\subsection{Sets}

Our signed sets will each involve $k$-tuples of graphs with edges colored either black or red (dashed). For ease of discussion, we will call the nodes $\{0,1,\dots,k\}$ \textbf{special nodes} and nodes $\{k+1,\dots,n\}$ \textbf{non-special nodes}.

Let $S_0$ be ordered sets of $k$ forests, the first being a tree rooted at node $0$, and the remaining $k-1$ being $(k+1)$-forests rooted at the special nodes. Note these are the sets involved in the left hand side of the Generalized Forest Identity. In this set, all edges will be black and all elements will be positive.

Let $S_1$ be ordered sets of $k$ graphs, the first with one edge out of every node except for node $0$, and the remaining $k-1$ graphs each with no edge out of the special nodes and one edge out of the rest of the nodes. All edges are colored black except for cycles, which can be colored red or black. If an even number of cycles in the set of graphs is red, then the set is positive. If an odd number of cycles is red, the set is negative.

Let $S_2$ be ordered sets of $k$ graphs, where the $i$th graph has one edge out of each non-special node and out of node $i$, and the remaining special nodes have no out-edges. In this case, we count any forbidden meta-cycle as a cycle. Each cycle can either be colored red or black. The rest of the edges must be colored black. If an even number of cycles is colored red, then the pair is positive. If an odd number of cycles is colored red, the pair is negative.

Let $S_3$ be ordered sets of $k$ non-forbidden $k$-forests, where the $i$th $k$-forest has roots at $0,\dots, i-1,i+1,\dots,k$. Note that these are the sets involved in the right hand side of the Generalized Forest Identity. In this set, all edges will be black and all elements will be positive.

Notice that when $k=2$, these sets are the sets defined in the original Red Hot Potato algorithm \cite{fraser}.

\subsection{Involutions}

We define involutions $\phi_0$, $\phi_1$, and $\phi_2$ on differences of sets and prove they are sign-reversing.

\subsubsection{$\phi_0$ and $\phi_2$}
We define $\phi_0:S_0-S_1\to S_0-S_1$ and $\phi_2:S_2-S_3\to S_2-S_3$ in the same way. Below, we will define $\phi_0$. We define $\phi_2$ in a similar way by replacing $S_0$ with $S_3$ and $S_1$ with $S_2$.
\begin{itemize}
\item Notice that $S_0\subset S_1$. If $t\in S_0\cap S_1$, then $\phi_0$ sends $t\in S_0$ to itself in $S_1$ and vice versa. This is clearly an involution. It is sign-reversing since $t$ is positive in both $S_0$ and $S_1$, so it is negative in $-S_1$.

\item If $t\in S_1$ and $t\notin S_0$, then there must be at least one cycle in the graphs. Then the involution changes the color of the cycle in the first graph in the ordered list that has a cycle (that is the graph with the edge out of the smallest node). If this graph has more than one cycle, the involution changes the color of the cycle involving the largest node. If there are no normal cycles but rather forbidden meta-cycles, we change the color of the forbidden meta-cycle involving the largest node (note that this only applies to $\phi_2$). Since we only have two colors, this is clearly an involution. It is also sign-reversing since changing the color of one cycle changes the parity of the number of red cycles.

\end{itemize}

Notice that when $k=2$, these involutions are the same $\phi_0$ and $\phi_2$ used in the Red Hot Potato algorithm \cite{fraser}.

\subsubsection{Crabwalk}
The involution $\phi_1:S_1-S_2\to S_1-S_2$ is the involution that actually moves edges back and forth between the graphs. When we talk about moving a red meta-edge $s\to a_1\to \dots\to a_r\to t$ from $A$ to $B$, we mean that we move all of the red edges in the $s\to t$ meta-edge from $A$ to $B$, and move all of the black edges in $B$ coming out of the nodes $a_1,a_2,\dots, a_r$ to $A$. We will begin by defining $\phi_1^i$, which moves edges between two graphs in our set of graphs. Let $\pi_{j}$ be the graph with an edge coming out of node $k$ (we will explain more below how to find the subscript $j$). Remove this graph from the list of graphs so that we now have an ordered sublist of $k-1$ graphs. Let the $i$th graph in the list be called $\tau_i$. Then $\phi_1^i$ will act on $\pi_{j}$ and $\tau_i$. We will call the result of $\phi_1^i(\{\pi_{j},\tau_i\})=\{\pi_{j\pm 1},\tau_i'\}$, where we choose the subscript of $\pi$ based on our list of graphs; this is explained more below. Let $m$ be the smallest node with $1\leq m\leq k$ such that there is an edge out of $m$ in one of the two graphs $\pi_j$ and $\tau_i$. We will call the graph that has an edge out of $m$ $A$ and will call the other graph $B$. We label as $F$ the graph with the edge out of $k$ and the other graph we label as $R$. We then perform the \textbf{crabwalk} using graphs $A$ and $B$. The crabwalk was defined in \cite{fraser}, and we will reproduce the definition here with minor adjustments.

\begin{itemize}
\item If the edge coming out of node $m$ is black, move that edge from $A$ to $B$ to form a new pair $\{C,D\}$ (where now $D$ has the edge out of node $m$).

\item If the edge coming out of node $m$ is red, then we move the edges designated by the \textbf{crabwalk}, defined as follows. Create a graph with the same node set as $F$, and with edge set the set of all red edges from $F$ and $R$. Color the edges coming from $F$ dark red and the edges coming from $R$ light red (dashed). This is the crabwalk colored graph. We will always move forward along dark red edges (edges from $F$), and backward along light red edges (edges from $R$). 

If the edge out of node $m$ is in $F$, then we begin by moving forward along the dark red edge coming out of $1$, changing its color to light red starting with the tail of the edge and then coloring the head. We continue along that meta-edge until we reach a node that has a light red edge going into it. We travel backwards along the light red meta-edge, changing first the head of the edge to dark red and then the tail to dark red, until we reach a node that has a dark red edge coming out of it. Then we travel forward along the dark red meta-edge, changing it to light red, until we reach a node that has a light red edge going into it. We continue in this manner until we have reached any of the nodes $0,\dots,k$.

If the edge out of node $m$ is in $R$, then let $v\in \{0,\dots,k\}$ be the node that $m$ is rooted at, that is to say that there is a meta-edge $m\to v$. Then we begin by moving backwards along the light red edge coming into node $v$, changing first the head of the edge to dark red and then the tail. We continue moving backwards along the meta-edge until we reach a node that has a dark red edge coming out of it. Then we travel forward along the dark red meta-edge, changing it to light red, and so on until we have reached any of the nodes $0,\dots,k$.

Returning to graphs $F$ and $R$, we move the red meta-edges that changed color in the crabwalk, so that all dark red edges are in $F$ and all light red ones are in $R$. Figure \ref{crabwalk} gives an example of the crabwalk.

\end{itemize}

Notice that $\phi_1^i$ is equivalent to $\phi_1$ from the original Red Hot Potato algorithm for all $i$ \cite{fraser}.

\begin{figure}
\begin{center}
\begin{tikzpicture}[scale=1.5]
\node (0) at(0,0) {0};
\node (1) at(1,.5) {1};
\node (2) at(2,0) {2};
\node (3) at(2,-1) {3};
\node (4) at(1,-1.5) {4};
\node (5) at(0,-1) {5};
\draw[thick, red, dashed, ->] (2) edge (3);
\draw[thick, red, dashed, ->] (3) edge (4);
\draw[thick, red, dashed, ->] (4) edge (1);

\node (0) at(4,0) {0};
\node (1) at(5,.5) {1};
\node (2) at(6,0) {2};
\node (3) at(6,-1) {3};
\node (4) at(5,-1.5) {4};
\node (5) at(4,-1) {5};
\draw[thick, red, dashed, ->] (1) edge (4);
\draw[thick, red, dashed, ->] (4) edge (2);

\end{tikzpicture} 

\vspace{5mm}

\hspace{69mm} $\downarrow$ Combine into colored crabwalk graph.

\vspace{5mm}

\begin{tikzpicture}[scale=1.5]
\node (0) at(0,0) {0};
\node (1) at(1,.5) {1};
\node (2) at(2,0) {2};
\node (3) at(2,-1) {3};
\node (4) at(1,-1.5) {4};
\node (5) at(0,-1) {5};
\draw[thick, black!50!red, ->] (2) edge (3);
\draw[thick, black!50!red, ->] (3) edge (4);
\draw[thick, black!50!red, bend left, ->] (4) edge (1);
\draw[thick, red!50, dashed, bend left, ->] (1) edge (4);
\draw[thick, red!50, dashed, ->] (4) edge (2);
\end{tikzpicture}

\vspace{5mm}

\hspace{33mm} $\downarrow$ Perform crabwalk.

\vspace{5mm}

\begin{tikzpicture}[scale=1.5]
\node (0) at(0,0) {0};
\node (1) at(1,.5) {1};
\node (2) at(2,0) {2};
\node (3) at(2,-1) {3};
\node (4) at(1,-1.5) {4};
\node (5) at(0,-1) {5};
\draw[thick, black!50!red, ->] (2) edge (3);
\draw[thick, black!50!red, ->] (3) edge (4);
\draw[thick, red!50, dashed, bend left, ->] (4) edge (1);
\draw[thick, red!50, dashed, bend left, ->] (1) edge (4);
\draw[thick, black!50!red, ->] (4) edge (2);
\end{tikzpicture}

\vspace{5mm}

\hspace{35mm} $\downarrow$ Separate back out.

\vspace{5mm}

\begin{tikzpicture}[scale=1.5]
\node (0) at(0,0) {0};
\node (1) at(1,.5) {1};
\node (2) at(2,0) {2};
\node (3) at(2,-1) {3};
\node (4) at(1,-1.5) {4};
\node (5) at(0,-1) {5};
\draw[thick, red, dashed, ->] (2) edge (3);
\draw[thick, red, dashed, ->] (3) edge (4);
\draw[thick, red, dashed, ->] (4) edge (2);

\node (0) at(4,0) {0};
\node (1) at(5,.5) {1};
\node (2) at(6,0) {2};
\node (3) at(6,-1) {3};
\node (4) at(5,-1.5) {4};
\node (5) at(4,-1) {5};
\draw[thick, red, dashed, bend left, ->] (1) edge (4);
\draw[thick, red, dashed, bend left, ->] (4) edge (1);

\end{tikzpicture} 
\caption{An illustration of the crabwalk on a pair of graphs from a $k=3$ set. All of the edges shown here are red meta-edges. The left-hand graph has an edge out of $3$, so we label it $F$ and the other graph $R$. The right-hand graph has an edge out of $m=1$, so because $m$ is in $R$, we begin by moving backwards along the meta-edge $1\to 2$, so we start by changing the color of the edge $4\to 2$. We then intersect the cycle from $F$, so we move forward along the edge $4\to 1$. \label{crabwalk}}
\end{center}
\end{figure}

\subsubsection{$\phi_1$}
Now let us put the $\phi_1^i$ together to make $\phi_1$. 
\begin{itemize}
\item If $a\in S_1$, then we define $\phi_1$ by first doing $\phi_1^1(\{\pi_1,\tau_1\})$, then $\phi_1^2(\{\pi_2,\tau_2\})$ (where $\pi_2$ is the result of $\phi_1^1$ and $\tau_2$ is from our original list of graphs), and so on until finally we do $\phi_1^{k-1}(\{\pi_{k-1},\tau_{k-1}\})$. Thus $\phi_1(\{\pi_1,\tau_1,\dots,\tau_{k-1}\})=\{\pi_{k},\tau_1',\dots,\tau_{k-1}'\}$. We are repeatedly changing $\pi$ in each iteration by using a new $\tau$ that was from our original list.

\item If $a\in S_2$, 
\begin{itemize}
\item If $\phi_1^i$ ends with a node in $\tau_i$ for all $1\leq i\leq k-1$, then we define $\phi_1$ by first doing $\phi_1^{k-1}(\{\pi_{k},\tau_{k-1}\})$, then $\phi_1^{k-2}(\{\pi_{k-1},\tau_{k-2}\})$ (where $\pi_{k-1}$ is the result of $\phi_1^{k-1}$ and $\tau_{k-2}$ is from our original list of graphs), and so on until finally we do $\phi_1^1(\{\pi_{2},\tau_1\})$. Thus $\phi_1(\{\pi_k,\tau_1,\dots,\tau_{k-1}\})=\{\pi_{1},\tau_1',\dots,\tau_{k-1}'\}$. We are repeatedly changing $\pi$ in each iteration by using a new $\tau$ that was from our original list.

\item If $\phi_1^i$ ends with a node in $\pi_{j+1}$ for any $j$, then we proceed as above until we reach the first $j$ for which this happens. Upon completing $\phi_1^j(\{\pi_{j+1},\tau_j\})$, we then perform $\phi_1^{j+1}(\{\pi_{j+1}', \tau_{j+1}'\})$ where $\pi_{j+1}'$ is the result of $\phi_1^j$ and $\tau_{j+1}'$ is the result of our first iteration of $\phi_1^{j+1}$. In effect, once we hit a $\phi_1^j$ that ends with a node in $\pi_{j+1}$, instead of continuing backwards down the $\phi_1^i$, we start going forwards back up the $\phi_1^i$. The result of $\phi_1^{j+1}(\{\pi_{j+1}', \tau_{j+1}'\})$ is $\{\pi_{j+2}',\tau_{j+1}''\}$. We keep going until we reach $\phi_1^{k-1}(\{\pi_{k-1}',\tau_{k-1}'\})$. Thus $$\phi_1(\{\pi_k,\tau_1,\dots,\tau_{k-1}\})=\{\pi_{k}',\tau_1,\dots,\tau_{j-1},\tau_j',\tau_{j+1}'',\dots,\tau_{k-1}''\}.$$ Notice that when we hit $\phi_1^j$, because we ended with a node in $\pi_{j+1}$, we must not have moved the edge out of $j$ from $\tau_j$ to $\pi_{j+1}$. Thus the resulting set of forests still has $\tau_i$ (with however many primes) having an edge out of node $i$ and $\pi_k'$ still having an edge out of $k$.
\end{itemize}
\end{itemize}

Notice that when $k=2$, $\phi_1$ is the same as $\phi_1$ from the original Red Hot Potato algorithm.

We will prove that $\phi_1$ is a sign-reversing involution in Section 5.

\subsection{Generalized Red Hot Potato algorithm}

It is shown in section 3.2.1 that both $\phi_0$ and $\phi_2$ are sign-reversing involutions. We claim that $\phi_1$ is also a sign-reversing involution, which we will prove in section 5. Then these sets and involutions satisfy the hypotheses of Theorem \ref{invothm}, so we have proved the Generalized Forest Identity. Following the algorithm for finding the bijection that the Involution Principle guarantees (see \cite{sally}), we can construct the bijection as follows:

We begin with a set of forests in $S_0$. We apply $\phi_0$ and then apply $\phi_1$. Then we change the color of the cycle dictated by $\phi_2$ and do $\phi_1$ again. Then we change the color of the appropriate cycle and so on. We finish when the output of $\phi_2$ is a set of forests in $S_3$.

\section{Example}

\begin{tikzpicture}[baseline=0ex]
\node (0) at(0,0) {0};
\node (1) at(1,.5) {1};
\node (2) at(2,0) {2};
\node (3) at(2,-1) {3};
\node (4) at(1,-1.5) {4};
\node (5) at(0,-1) {5};
\draw[thick, ->] (1) edge (5);
\draw[thick, ->] (2) edge (0);
\draw[thick, ->] (3) edge (4);
\draw[thick, ->] (4) edge (1);
\draw[thick, ->] (5) edge (0);

\node (0) at(4,0) {0};
\node (1) at(5,.5) {1};
\node (2) at(6,0) {2};
\node (3) at(6,-1) {3};
\node (4) at(5,-1.5) {4};
\node (5) at(4,-1) {5};
\draw[thick, ->] (4) edge (3);
\draw[thick, ->] (5) edge (4);

\node (0) at(8,0) {0};
\node (1) at(9,.5) {1};
\node (2) at(10,0) {2};
\node (3) at(10,-1) {3};
\node (4) at(9,-1.5) {4};
\node (5) at(8,-1) {5};
\draw[thick, ->] (4) edge (2);
\draw[thick, ->] (5) edge (2);
\end{tikzpicture} 

We begin by applying $\phi_1^1$ to the left two graphs. Since all the edges are black, we only move the one coming out of 1.

\vspace{10mm}

\begin{tikzpicture}[baseline=0ex]
\node (0) at(0,0) {0};
\node (1) at(1,.5) {1};
\node (2) at(2,0) {2};
\node (3) at(2,-1) {3};
\node (4) at(1,-1.5) {4};
\node (5) at(0,-1) {5};
\draw[thick, ->] (2) edge (0);
\draw[thick, ->] (3) edge (4);
\draw[thick, ->] (4) edge (1);
\draw[thick, ->] (5) edge (0);

\node (0) at(4,0) {0};
\node (1) at(5,.5) {1};
\node (2) at(6,0) {2};
\node (3) at(6,-1) {3};
\node (4) at(5,-1.5) {4};
\node (5) at(4,-1) {5};
\draw[thick, ->] (1) edge (5);
\draw[thick, ->] (4) edge (3);
\draw[thick, ->] (5) edge (4);

\node (0) at(8,0) {0};
\node (1) at(9,.5) {1};
\node (2) at(10,0) {2};
\node (3) at(10,-1) {3};
\node (4) at(9,-1.5) {4};
\node (5) at(8,-1) {5};
\draw[thick, ->] (4) edge (2);
\draw[thick, ->] (5) edge (2);
\end{tikzpicture} 

We now do the same to the left-most graph ($\pi$) and the right-most one.

\vspace{10mm}

\begin{tikzpicture}[baseline=0ex]
\node (0) at(0,0) {0};
\node (1) at(1,.5) {1};
\node (2) at(2,0) {2};
\node (3) at(2,-1) {3};
\node (4) at(1,-1.5) {4};
\node (5) at(0,-1) {5};
\draw[thick, ->] (3) edge (4);
\draw[thick, ->] (4) edge (1);
\draw[thick, ->] (5) edge (0);

\node (0) at(4,0) {0};
\node (1) at(5,.5) {1};
\node (2) at(6,0) {2};
\node (3) at(6,-1) {3};
\node (4) at(5,-1.5) {4};
\node (5) at(4,-1) {5};
\draw[thick, ->] (1) edge (5);
\draw[thick, ->] (4) edge (3);
\draw[thick, ->] (5) edge (4);

\node (0) at(8,0) {0};
\node (1) at(9,.5) {1};
\node (2) at(10,0) {2};
\node (3) at(10,-1) {3};
\node (4) at(9,-1.5) {4};
\node (5) at(8,-1) {5};
\draw[thick, ->] (2) edge (0);
\draw[thick, ->] (4) edge (2);
\draw[thick, ->] (5) edge (2);
\end{tikzpicture} 

We have finished with $\phi_1$. Now we notice that there is a meta-cycle from $1\to 3$ and $3\to 1$, so we change the color to red (dashed).

\vspace{10mm}

\begin{tikzpicture}[baseline=0ex]
\node (0) at(0,0) {0};
\node (1) at(1,.5) {1};
\node (2) at(2,0) {2};
\node (3) at(2,-1) {3};
\node (4) at(1,-1.5) {4};
\node (5) at(0,-1) {5};
\draw[thick, red, dashed, ->] (3) edge (4);
\draw[thick, red, dashed, ->] (4) edge (1);
\draw[thick, ->] (5) edge (0);

\node (0) at(4,0) {0};
\node (1) at(5,.5) {1};
\node (2) at(6,0) {2};
\node (3) at(6,-1) {3};
\node (4) at(5,-1.5) {4};
\node (5) at(4,-1) {5};
\draw[thick, red, dashed, ->] (1) edge (5);
\draw[thick, red, dashed, ->] (4) edge (3);
\draw[thick, red, dashed, ->] (5) edge (4);

\node (0) at(8,0) {0};
\node (1) at(9,.5) {1};
\node (2) at(10,0) {2};
\node (3) at(10,-1) {3};
\node (4) at(9,-1.5) {4};
\node (5) at(8,-1) {5};
\draw[thick, ->] (2) edge (0);
\draw[thick, ->] (4) edge (2);
\draw[thick, ->] (5) edge (2);
\end{tikzpicture} 

We now begin $\phi_1$ again, this time starting with the left-most graph ($\pi$) and the right-most one. Since the edge out of $2$ is black, we simply move that edge over.

\vspace{10mm}

\begin{tikzpicture}[baseline=0ex]
\node (0) at(0,0) {0};
\node (1) at(1,.5) {1};
\node (2) at(2,0) {2};
\node (3) at(2,-1) {3};
\node (4) at(1,-1.5) {4};
\node (5) at(0,-1) {5};
\draw[thick, ->] (2) edge (0);
\draw[thick, red, dashed, ->] (3) edge (4);
\draw[thick, red, dashed, ->] (4) edge (1);
\draw[thick, ->] (5) edge (0);

\node (0) at(4,0) {0};
\node (1) at(5,.5) {1};
\node (2) at(6,0) {2};
\node (3) at(6,-1) {3};
\node (4) at(5,-1.5) {4};
\node (5) at(4,-1) {5};
\draw[thick, red, dashed, ->] (1) edge (5);
\draw[thick, red, dashed, ->] (4) edge (3);
\draw[thick, red, dashed, ->] (5) edge (4);

\node (0) at(8,0) {0};
\node (1) at(9,.5) {1};
\node (2) at(10,0) {2};
\node (3) at(10,-1) {3};
\node (4) at(9,-1.5) {4};
\node (5) at(8,-1) {5};
\draw[thick, ->] (4) edge (2);
\draw[thick, ->] (5) edge (2);
\end{tikzpicture} 

Now we apply $\phi_1$ to the left-most graph and the middle one. When we apply the crabwalk, we first move backwards along the edge into node $3$ in the middle graph (because $3$ is the root of node $1$), and then forwards along the edge out of node $4$ in the first graph.

\vspace{10mm}

\begin{tikzpicture}[baseline=0ex]
\node (0) at(0,0) {0};
\node (1) at(1,.5) {1};
\node (2) at(2,0) {2};
\node (3) at(2,-1) {3};
\node (4) at(1,-1.5) {4};
\node (5) at(0,-1) {5};
\draw[thick, ->] (2) edge (0);
\draw[thick, red, dashed, bend left, ->] (3) edge (4);
\draw[thick, red, dashed, bend left, ->] (4) edge (3);
\draw[thick, ->] (5) edge (0);

\node (0) at(4,0) {0};
\node (1) at(5,.5) {1};
\node (2) at(6,0) {2};
\node (3) at(6,-1) {3};
\node (4) at(5,-1.5) {4};
\node (5) at(4,-1) {5};
\draw[thick, red, dashed, ->] (1) edge (5);
\draw[thick, red, dashed, ->] (4) edge (1);
\draw[thick, red, dashed, ->] (5) edge (4);

\node (0) at(8,0) {0};
\node (1) at(9,.5) {1};
\node (2) at(10,0) {2};
\node (3) at(10,-1) {3};
\node (4) at(9,-1.5) {4};
\node (5) at(8,-1) {5};
\draw[thick, ->] (4) edge (2);
\draw[thick, ->] (5) edge (2);
\end{tikzpicture} 

Notice that we ended that last iteration of $\phi_1^1$ in the left-most graph, $\pi$. That means that instead of being done with $\phi_1$, we move forward again by doing $\phi_1^2$ on the left-most graph and the right-most one.

\vspace{10mm}

\begin{tikzpicture}[baseline=0ex]
\node (0) at(0,0) {0};
\node (1) at(1,.5) {1};
\node (2) at(2,0) {2};
\node (3) at(2,-1) {3};
\node (4) at(1,-1.5) {4};
\node (5) at(0,-1) {5};
\draw[thick, red, dashed, bend left, ->] (3) edge (4);
\draw[thick, red, dashed, bend left, ->] (4) edge (3);
\draw[thick, ->] (5) edge (0);

\node (0) at(4,0) {0};
\node (1) at(5,.5) {1};
\node (2) at(6,0) {2};
\node (3) at(6,-1) {3};
\node (4) at(5,-1.5) {4};
\node (5) at(4,-1) {5};
\draw[thick, red, dashed, ->] (1) edge (5);
\draw[thick, red, dashed, ->] (4) edge (1);
\draw[thick, red, dashed, ->] (5) edge (4);

\node (0) at(8,0) {0};
\node (1) at(9,.5) {1};
\node (2) at(10,0) {2};
\node (3) at(10,-1) {3};
\node (4) at(9,-1.5) {4};
\node (5) at(8,-1) {5};
\draw[thick, ->] (2) edge (0);
\draw[thick, ->] (4) edge (2);
\draw[thick, ->] (5) edge (2);
\end{tikzpicture} 

We have now finished with $\phi_1$. Then we change the color of the cycle in the graph with an edge out of the smallest node, namely the cycle in the middle graph involving node $1$.

\vspace{10mm}

\begin{tikzpicture}[baseline=0ex]
\node (0) at(0,0) {0};
\node (1) at(1,.5) {1};
\node (2) at(2,0) {2};
\node (3) at(2,-1) {3};
\node (4) at(1,-1.5) {4};
\node (5) at(0,-1) {5};
\draw[thick, red, dashed, bend left, ->] (3) edge (4);
\draw[thick, red, dashed, bend left, ->] (4) edge (3);
\draw[thick, ->] (5) edge (0);

\node (0) at(4,0) {0};
\node (1) at(5,.5) {1};
\node (2) at(6,0) {2};
\node (3) at(6,-1) {3};
\node (4) at(5,-1.5) {4};
\node (5) at(4,-1) {5};
\draw[thick, ->] (1) edge (5);
\draw[thick, ->] (4) edge (1);
\draw[thick, ->] (5) edge (4);

\node (0) at(8,0) {0};
\node (1) at(9,.5) {1};
\node (2) at(10,0) {2};
\node (3) at(10,-1) {3};
\node (4) at(9,-1.5) {4};
\node (5) at(8,-1) {5};
\draw[thick, ->] (2) edge (0);
\draw[thick, ->] (4) edge (2);
\draw[thick, ->] (5) edge (2);
\end{tikzpicture} 

We now do $\phi_1$ again. We start by doing $\phi_1^2$ with the left-most and right-most graphs.

\vspace{10mm}

\begin{tikzpicture}[baseline=0ex]
\node (0) at(0,0) {0};
\node (1) at(1,.5) {1};
\node (2) at(2,0) {2};
\node (3) at(2,-1) {3};
\node (4) at(1,-1.5) {4};
\node (5) at(0,-1) {5};
\draw[thick, ->] (2) edge (0);
\draw[thick, red, dashed, bend left, ->] (3) edge (4);
\draw[thick, red, dashed, bend left, ->] (4) edge (3);
\draw[thick, ->] (5) edge (0);

\node (0) at(4,0) {0};
\node (1) at(5,.5) {1};
\node (2) at(6,0) {2};
\node (3) at(6,-1) {3};
\node (4) at(5,-1.5) {4};
\node (5) at(4,-1) {5};
\draw[thick, ->] (1) edge (5);
\draw[thick, ->] (4) edge (1);
\draw[thick, ->] (5) edge (4);

\node (0) at(8,0) {0};
\node (1) at(9,.5) {1};
\node (2) at(10,0) {2};
\node (3) at(10,-1) {3};
\node (4) at(9,-1.5) {4};
\node (5) at(8,-1) {5};
\draw[thick, ->] (4) edge (2);
\draw[thick, ->] (5) edge (2);
\end{tikzpicture} 

Now we do $\phi_1^1$ with the left-most and middle graphs.

\vspace{10mm}

\begin{tikzpicture}[baseline=0ex]
\node (0) at(0,0) {0};
\node (1) at(1,.5) {1};
\node (2) at(2,0) {2};
\node (3) at(2,-1) {3};
\node (4) at(1,-1.5) {4};
\node (5) at(0,-1) {5};
\draw[thick, ->] (1) edge (5);
\draw[thick, ->] (2) edge (0);
\draw[thick, red, dashed, bend left, ->] (3) edge (4);
\draw[thick, red, dashed, bend left, ->] (4) edge (3);
\draw[thick, ->] (5) edge (0);

\node (0) at(4,0) {0};
\node (1) at(5,.5) {1};
\node (2) at(6,0) {2};
\node (3) at(6,-1) {3};
\node (4) at(5,-1.5) {4};
\node (5) at(4,-1) {5};
\draw[thick, ->] (4) edge (1);
\draw[thick, ->] (5) edge (4);

\node (0) at(8,0) {0};
\node (1) at(9,.5) {1};
\node (2) at(10,0) {2};
\node (3) at(10,-1) {3};
\node (4) at(9,-1.5) {4};
\node (5) at(8,-1) {5};
\draw[thick, ->] (4) edge (2);
\draw[thick, ->] (5) edge (2);
\end{tikzpicture} 

We have now finished with $\phi_1$. We change the color of the remaining cycle.

\vspace{10mm}

\begin{tikzpicture}[baseline=0ex]
\node (0) at(0,0) {0};
\node (1) at(1,.5) {1};
\node (2) at(2,0) {2};
\node (3) at(2,-1) {3};
\node (4) at(1,-1.5) {4};
\node (5) at(0,-1) {5};
\draw[thick, ->] (1) edge (5);
\draw[thick, ->] (2) edge (0);
\draw[thick, bend left, ->] (3) edge (4);
\draw[thick, bend left, ->] (4) edge (3);
\draw[thick, ->] (5) edge (0);

\node (0) at(4,0) {0};
\node (1) at(5,.5) {1};
\node (2) at(6,0) {2};
\node (3) at(6,-1) {3};
\node (4) at(5,-1.5) {4};
\node (5) at(4,-1) {5};
\draw[thick, ->] (4) edge (1);
\draw[thick, ->] (5) edge (4);

\node (0) at(8,0) {0};
\node (1) at(9,.5) {1};
\node (2) at(10,0) {2};
\node (3) at(10,-1) {3};
\node (4) at(9,-1.5) {4};
\node (5) at(8,-1) {5};
\draw[thick, ->] (4) edge (2);
\draw[thick, ->] (5) edge (2);
\end{tikzpicture} 

Now we perform $\phi_1$ again, first moving the black edge out of $1$ in $\phi_1^1$, then the black edge out of $2$ in $\phi_1^2$.

\vspace{10mm}

\begin{tikzpicture}[baseline=0ex]
\node (0) at(0,0) {0};
\node (1) at(1,.5) {1};
\node (2) at(2,0) {2};
\node (3) at(2,-1) {3};
\node (4) at(1,-1.5) {4};
\node (5) at(0,-1) {5};
\draw[thick, bend left, ->] (3) edge (4);
\draw[thick, bend left, ->] (4) edge (3);
\draw[thick, ->] (5) edge (0);

\node (0) at(4,0) {0};
\node (1) at(5,.5) {1};
\node (2) at(6,0) {2};
\node (3) at(6,-1) {3};
\node (4) at(5,-1.5) {4};
\node (5) at(4,-1) {5};
\draw[thick, ->] (1) edge (5);
\draw[thick, ->] (4) edge (1);
\draw[thick, ->] (5) edge (4);

\node (0) at(8,0) {0};
\node (1) at(9,.5) {1};
\node (2) at(10,0) {2};
\node (3) at(10,-1) {3};
\node (4) at(9,-1.5) {4};
\node (5) at(8,-1) {5};
\draw[thick, ->] (2) edge (0);
\draw[thick, ->] (4) edge (2);
\draw[thick, ->] (5) edge (2);
\end{tikzpicture} 

Done with $\phi_1$, we now change the color of the cycle in the graph with an edge out of the smallest node, namely the cycle in the middle graph.

\vspace{10mm}

\begin{tikzpicture}[baseline=0ex]
\node (0) at(0,0) {0};
\node (1) at(1,.5) {1};
\node (2) at(2,0) {2};
\node (3) at(2,-1) {3};
\node (4) at(1,-1.5) {4};
\node (5) at(0,-1) {5};
\draw[thick, bend left, ->] (3) edge (4);
\draw[thick, bend left, ->] (4) edge (3);
\draw[thick, ->] (5) edge (0);

\node (0) at(4,0) {0};
\node (1) at(5,.5) {1};
\node (2) at(6,0) {2};
\node (3) at(6,-1) {3};
\node (4) at(5,-1.5) {4};
\node (5) at(4,-1) {5};
\draw[thick, red, dashed, ->] (1) edge (5);
\draw[thick, red, dashed, ->] (4) edge (1);
\draw[thick, red, dashed, ->] (5) edge (4);

\node (0) at(8,0) {0};
\node (1) at(9,.5) {1};
\node (2) at(10,0) {2};
\node (3) at(10,-1) {3};
\node (4) at(9,-1.5) {4};
\node (5) at(8,-1) {5};
\draw[thick, ->] (2) edge (0);
\draw[thick, ->] (4) edge (2);
\draw[thick, ->] (5) edge (2);
\end{tikzpicture} 

When we apply $\phi_1$, we first apply $\phi_1^2$ to move the black edge out of $2$ to the left-most graph. Then we apply $\phi_1^1$ to the left-most and middle graphs, moving the entirety of the red cycle, as well as the black edges out of nodes $4$ and $5$ along the way.

\vspace{10mm}

\begin{tikzpicture}[baseline=0ex]
\node (0) at(0,0) {0};
\node (1) at(1,.5) {1};
\node (2) at(2,0) {2};
\node (3) at(2,-1) {3};
\node (4) at(1,-1.5) {4};
\node (5) at(0,-1) {5};
\draw[thick, red, dashed, ->] (1) edge (5);
\draw[thick, ->] (2) edge (0);
\draw[thick, ->] (3) edge (4);
\draw[thick, red, dashed, ->] (4) edge (1);
\draw[thick, red, dashed, ->] (5) edge (4);

\node (0) at(4,0) {0};
\node (1) at(5,.5) {1};
\node (2) at(6,0) {2};
\node (3) at(6,-1) {3};
\node (4) at(5,-1.5) {4};
\node (5) at(4,-1) {5};
\draw[thick, ->] (4) edge (3);
\draw[thick, ->] (5) edge (0);

\node (0) at(8,0) {0};
\node (1) at(9,.5) {1};
\node (2) at(10,0) {2};
\node (3) at(10,-1) {3};
\node (4) at(9,-1.5) {4};
\node (5) at(8,-1) {5};
\draw[thick, ->] (4) edge (2);
\draw[thick, ->] (5) edge (2);
\end{tikzpicture} 

Done with $\phi_1$, we now change the color of the cycle in the left-most graph.

\vspace{10mm}

\begin{tikzpicture}[baseline=0ex]
\node (0) at(0,0) {0};
\node (1) at(1,.5) {1};
\node (2) at(2,0) {2};
\node (3) at(2,-1) {3};
\node (4) at(1,-1.5) {4};
\node (5) at(0,-1) {5};
\draw[thick, ->] (1) edge (5);
\draw[thick, ->] (2) edge (0);
\draw[thick, ->] (3) edge (4);
\draw[thick, ->] (4) edge (1);
\draw[thick, ->] (5) edge (4);

\node (0) at(4,0) {0};
\node (1) at(5,.5) {1};
\node (2) at(6,0) {2};
\node (3) at(6,-1) {3};
\node (4) at(5,-1.5) {4};
\node (5) at(4,-1) {5};
\draw[thick, ->] (4) edge (3);
\draw[thick, ->] (5) edge (0);

\node (0) at(8,0) {0};
\node (1) at(9,.5) {1};
\node (2) at(10,0) {2};
\node (3) at(10,-1) {3};
\node (4) at(9,-1.5) {4};
\node (5) at(8,-1) {5};
\draw[thick, ->] (4) edge (2);
\draw[thick, ->] (5) edge (2);
\end{tikzpicture} 

We apply $\phi_1$ once more, first moving the black edge out of $1$ with $\phi_1^1$, and then the black edge out of $2$ with $\phi_1^2$.

\vspace{10mm}

\begin{tikzpicture}[baseline=0ex]
\node (0) at(0,0) {0};
\node (1) at(1,.5) {1};
\node (2) at(2,0) {2};
\node (3) at(2,-1) {3};
\node (4) at(1,-1.5) {4};
\node (5) at(0,-1) {5};
\draw[thick, ->] (3) edge (4);
\draw[thick, ->] (4) edge (1);
\draw[thick, ->] (5) edge (4);

\node (0) at(4,0) {0};
\node (1) at(5,.5) {1};
\node (2) at(6,0) {2};
\node (3) at(6,-1) {3};
\node (4) at(5,-1.5) {4};
\node (5) at(4,-1) {5};
\draw[thick, ->] (1) edge (5);
\draw[thick, ->] (4) edge (3);
\draw[thick, ->] (5) edge (0);

\node (0) at(8,0) {0};
\node (1) at(9,.5) {1};
\node (2) at(10,0) {2};
\node (3) at(10,-1) {3};
\node (4) at(9,-1.5) {4};
\node (5) at(8,-1) {5};
\draw[thick, ->] (2) edge (0);
\draw[thick, ->] (4) edge (2);
\draw[thick, ->] (5) edge (2);
\end{tikzpicture} 

When we apply $\phi_2$ to this, we get the same thing back again. We now have an appropriate set of forests, so we are done!

\section{Proof that $\phi_1$ is a sign-reversing involution}

Because each $\phi_1^i$ is equivalent to $\phi_1$ from the original Red Hot Potato algorithm, we will restate below several Lemmas whose proofs can be found in \cite{fraser}. These Lemmas were originally about $\phi_1$ from the Red Hot Potato algorithm, but will be rephrased here to apply to each individual $\phi_1^i$.

\begin{lemma}\label{Range}
Applying $\phi_1^i$ to $(\pi_j,\tau_i)$ will result in a pair of graphs $(C,D)$ where each graph has exactly one edge out of each of nodes $k+1,\dots,n$, and all red meta-edges not involving special nodes must be cycles.
\end{lemma}

\begin{lemma}\label{involution}
Each function $\phi_1^i$ is an involution.
\end{lemma}

\begin{lemma}\label{Parity}
If the crabwalk ends in graph $A$, then the parity of cycles (whether there are an odd or even number of cycles) remains the same after applying $\phi_1^i$. If the crabwalk ends in graph $B$, then the parity of cycles changes after applying $\phi_1^i$.
\end{lemma}

We will now use the above Lemmas to inform our proof that $\phi_1$ in the Generalized Red Hot Potato algorithm is a sign-reversing involution.

\begin{lemma}\label{genRange}
Let $(G_1,G_2,\dots,G_k)\in S_1-S_2$. Then $\phi_1((G_1,G_2,\dots,G_k))\in S_1-S_2$.
\end{lemma}

\begin{proof}
To prove that $\phi_1((G_1,G_2,\dots,G_k))=(G_1',G_2',\dots,G_k')\in S_1-S_2$, we need to show that $(G_1',G_2',\dots,G_k')$ has the following two defining characteristics:
\begin{enumerate}
\item There are no edges out of node $0$ in any graph. There is exactly one edge out of each of the nodes $1,\dots,k$, and they are either all in one graph, or exactly one of these $k$ edges is in each graph. There are $k$ edges out of each of the remaining nodes $k+1,\dots,n$, exactly one in each of the $k$ graphs.
\item Red edges are only in cycles (this includes forbidden meta-cycles).
\end{enumerate}

We begin with requirement (a). Since $\phi_1$ simply moves edges around, and does not add or delete any edges, then since we have started with no edges out of node 0, one edge out of nodes $1,\dots,k$, and $k$ edges out of the rest, we will end with that as well. By Lemma \ref{Range}, each $\phi_1^i$ ends with one edge out nodes $k+1,\dots,n$ in each graph for all $i$. Then since $\phi_1$ is merely iterations of $\phi_1^i$, $\phi_1$ also ends with one edge out of nodes $k+1,\dots,n$ in each graph.

To finish with requirement (a), we must show that the single edges out of each of the nodes $1,\dots,k$ will either all end up in the same graph, or all end up in different graphs.

\vspace{5mm}

\textbf{Case 1}
Suppose $(G_1,G_2,\dots,G_k)\in S_1$, so for each $i$, $\tau_i$ has no edges out of the nodes $1\dots,k$. Then in this case for the crabwalk, $A$ and $F$ are both $\pi$. Since we are performing the $\phi_1^i$ in increasing order, the smallest node with an edge out of it will be $i$, and since we begin each $\phi_1^i$ with $\pi$, we will be moving an edge out of $i$ from $\pi$ to $\tau_i$ in the first step. Because there are no edges out of $1,\dots, k$ in any of the $\tau_i$, we must end in $\pi$ with an edge going into one of the special nodes. Thus the edge out of $i$ is the only edge out of a special node that is moved. Thus at the end of $\phi_1$, for all $1\leq i\leq k-1$, $\tau_i'$ will have an edge coming out of $i$, and $\pi$ will have an edge coming out of $k$, so $\phi_1((G_1,G_2,\dots,G_k))=(G_1',G_2',\dots,G_k')\in S_2$ (provided the red edges are in cycles).

\vspace{5mm}

\textbf{Case 2}
Suppose $(G_1,G_2,\dots,G_k)\in S_2$. Since we are moving backwards along iterations of $\phi_1^i$, then when we start, $\tau_i$ will have an edge out of $i$ and $\pi$ will have edges out of larger nodes. Thus for the beginning of $\phi_1$, $m=i$ and in the crabwalk, $\tau_i$ is $A$ and $R$ and $\pi$ is $B$ and $F$. Then we can break this into two further cases:
\begin{enumerate}[label=(\alph*)]
\item Let $\phi_1^i$ end with a node in $\tau_i$ for all $i$. Then since we move backwards in $\tau_i$, we must finish by moving an edge out of a special node. But the only special node with an edge out of it in $\tau_i$ is $i$. Thus we have moved the edge out of $i$ from $\tau_i$ to $\pi$ for all $i$, so when we have finished with $\phi_1$, all edges out of nodes $1,\dots,k$ will be in $\pi$ (and none will be in $\tau_i'$). Then $\phi_1((G_1,G_2,\dots,G_k))=(G_1',G_2',\dots,G_k')\in S_1$ (provided the red edges are in cycles).

\item Let there exist an $i$ for which $\phi_1^i$ ends with a node in $\pi$, and let $\phi_1^j$ be the first time this happens. Then up until $\phi_1^j$, we have proceeded as above. This means that, before we do $\phi_1^j$:

\begin{itemize}
\item $\pi$ has edges out of nodes $j+1,\dots,k$

\item $\tau_i'$ for $i>j$ have no edges out of nodes $1,\dots,k$

\item $\tau_i$ for $i\leq j$ have one edge out of $i$. 
\end{itemize}

Then when we perform $\phi_1^j(\pi,\tau_j)$, we end at a node in $\pi$. That means that we have finished by moving forward along an edge into one of the nodes $1,\dots,k$, so we have not reached the edge out of $j$ in $\tau_j$. Thus the edge out of $j$ remains in $\tau_j'$, and the edges out of $j+1,\cdots k$ are all in $\pi$. Then when we move forwards again, we are in Case 1, so the edges out of $i\in\{j+1,\dots, k-1\}$ get moved out into $\tau_i''$ for all $i$, leaving us with an edge out of $k$ in $\pi$. Thus $\phi_1((G_1,G_2,\dots,G_k))=(G_1',G_2',\dots,G_k')\in S_2$ (provided the red edges are in cycles).
\end{enumerate}

We now prove requirement (b), that the red edges will only be involved in cycles. Since $\phi_1$ is made up of iterations of $\phi_1^i$, it suffices to show that, if we plug in two graphs $\{A,B\}$ with red edges only involved in cycles, $\phi_1^i$ gives us two new graphs $\{C,D\}$ with red edges only involved in cycles. By Lemma \ref{Range}, all red edges not involving the nodes $1,\dots k$ are involved in cycles.

Let us now move on to the red edges involving our special nodes. Because we started with a set of graphs in $S_1-S_2$, if any special node $i$ has a red edge out of it, that node must originally have been part of a cycle or meta-cycle, so across all graphs, there must be exactly one red edge pointed into node $i$. Because $\phi_1$ only moves edges around, this is still true after applying $\phi_1$. If the red edge into and out of a special node $i$ are both in the same graph, then this node will be involved in a cycle: If it were not involved in a cycle, then the meta-edge passing through $i$ would need a place to start, that is a node $j$ where the edge leaving $j$ were red, but the edge entering $j$ was not red. This could only happen at a special node, since non-special nodes can only have red edges involved in cycles. If we were in the edge configuration for $S_2$, this is impossible, because there is only an edge out of node $i$ in this graph, so we could not also have a red edge out of special node $j$. If we were in the edge configuration for $S_1$, then the only way for this to happen would be if the red edge going into $j$ were in a different graph. However, in that case, that red meta-edge pointing into $j$ would need to start somewhere, and because non-special nodes can only have red edges in if there are also red edges out, it would need to start with an edge coming out of a special node. However, in $S_1$, all edges out of special nodes are in the same graph, so there would be nowhere for this meta-edge to start. Thus if $i$ has a red edge both into and out of it in the same graph, then $i$ is part of a red cycle.

Now suppose that $i$ has the red edge into it in one graph and out of it in another graph. We want to show that this means that $i$ is part of a meta-cycle. As seen in the previous paragraph, if the red edges into and out of $i$ are in two different graphs, then we must be in the edge configuration of $S_2$. Since whenever non-special nodes have red edges into them they must also have red edges out, then the red meta-edge out of $i$ must end with a special node $j$. This special node $j$ then has a red edge into it, but the red edge out must be in a different graph since we are in $S_2$. We can continue in this vein, creating a unique sequence of maximal red meta-edges $i\to j,j\to a_1,a_1\to a_2,\dots$ where each $a_i\in\{1,\dots k\}$. We want to show that at some point, the sequence loops back around to $i$, because then the sequence would form a meta-cycle. Let us look at the first $k$ terms of the sequence: $i\to j,j\to a_1,\dots,a_{k-2}\to a_{k-1}$. Then we have $k+1$ nodes represented in these sequences, and we have $k$ nodes to choose them from. By the pigeon hole principle, we must have at least two nodes in our list that are the same. Let $a_v$ be the first node in the list that repeats one we have seen before. Then our list looks like $i\to j, j\to a_1,\dots,a_{v-1}\to a_v$ where $a_v\in\{i,j,a_1,\dots,a_{v-1}\}$. Suppose $a_v\neq i$. Since there is only one red edge into $a_v$, and our sequence consists of maximal red meta-edges, then the entire meta-edge $a_{v-1}\to a_v$ must have shown up in our sequence before. But then $a_{v-1}$ is also a node that we have seen before, contradicting the assumption that $a_v$ was the first. Thus $a_v=i$ and our sequence forms a meta-cycle.

We have now proved that $\phi_1((G_1,\dots,G_k))$ satisfies both conditions required for it to be in $S_1-S_2$. Thus $\phi_1$ is in fact a function from $S_1-S_2$ to $S_1-S_2$.

\end{proof}

\begin{theorem}
The function $\phi_1$ is a sign-reversing involution on $S_1-S_2$.
\end{theorem}

\begin{proof}

Lemma \ref{genRange} shows that $\phi_1$ is indeed a function into the correct range. By Lemma \ref{involution}, $\phi_1^i$ is an involution for all $i$. We break our proof that $\phi_1$ is a sign-reversing involution into two cases:

\vspace{5mm}

\textbf{Case 1} Let $(G_1,G_2,\dots G_k)\in S_1$. As observed in the proof of Lemma \ref{genRange}, each iteration of $\phi_1^i$ must end in $A$. Then by Lemma \ref{Parity}, after each iteration of $\phi_1^i$, the parity of cycles remains the same, so after applying $\phi_1$, the parity of the cycles remains the same. As we saw in the proof of Lemma \ref{genRange}, $\phi_1((G_1,G_2,\dots,G_k))=(G_1',G_2',\dots,G_k')\in S_2$. Then since the parity of cycles is the same after applying $\phi_1$, $(G_1',G_2',\dots,G_k')\in S_2$ has the same sign as $(G_1,G_2,\dots G_k)\in S_1$. Thus the sign has changed in $S_1-S_2$. When we apply $\phi_1$ to $(G_1',G_2',\dots,G_k')\in S_2$, where we started with $(G_1,G_2,\dots,G_k)\in S_1$, we perform the $\phi_1^i$ in the opposite order than we did going from $S_1$ to $S_2$. As observed above, $\phi_1^i$ is an involution for all $i$, so $\phi_1(G_1',G_2',\dots,G_k')=(G_1,G_2,\dots,G_k)$. Thus, in this case, $\phi_1$ is a sign-reversing involution.

\vspace{5mm}

\textbf{Case 2} Let $(G_1,G_2,\dots G_k)\in S_2$.
\begin{enumerate}[label=(\alph*)]
\item Let $\phi_1^i$ end with a node in $\tau_i$ for all $i$. Then as we saw in the proof of Lemma \ref{genRange}, each iteration of $\phi_1^i$ must end in $A$. Then by Lemma \ref{Parity}, after each iteration of $\phi_1^i$, the parity of cycles remains the same, so after applying $\phi_1$, the parity of the cycles remains the same. As we saw in the proof of Lemma \ref{genRange}, $\phi_1((G_1,G_2,\dots,G_k))=(G_1',G_2',\dots,G_k')\in S_1$. Then since the parity of cycles is the same after applying $\phi_1$, $(G_1',G_2',\dots,G_k')\in S_1$ has the same sign as $(G_1,G_2,\dots G_k)\in S_2$. Thus the sign has changed in $S_1-S_2$. Similar to Case 1, when we apply $\phi_1$ to $(G_1',G_2',\dots,G_k')\in S_1$, we are applying the involutions $\phi_1^i$ in the opposite order, so we are undoing each step of our original $\phi_1$. Thus in this case, $\phi_1$ is a sign-reversing involution.

\item Let there exist an $i$ for which $\phi_1^i$ ends with a node in $\pi$, and let $\phi_1^j$ be the first time this happens. Then, as we saw in the proof of Lemma \ref{genRange}, up until performing $\phi_1^j$, we have ended in $A$, so until we apply $\phi_1^j$, the parity of the cycles remains the same. When we apply $\phi_1^j$, we end in $B$, so the parity of cycles switches. Then, as we move back up the $\phi_1^i$, we again end in $A$ each time, so the parity remains the same. Thus, once we have finished with $\phi_1$, we have ended in $B$ exactly one time, so the parity of cycles has changed. As we saw in the proof of Lemma \ref{genRange}, $\phi_1((G_1,G_2,\dots,G_k))=(G_1',G_2',\dots,G_k')\in S_2$. Then, since the parity of cycles has changed, the sign has as well.

Since after performing $\phi_1$, the graphs remain in the same order, then when we apply $\phi_1$ to $(G_1',G_2',\dots,G_k')\in S_2$, we are applying the involutions $\phi_1^i$ to the same sets of graphs in the opposite order, again undoing each step of the original $\phi_1$. Thus, again, $\phi_1$ is a sign-reversing involution.
\end{enumerate}

\end{proof}

\section{Mathematica Code Example}

The author programmed the generalized Red Hot Potato algorithm into Mathematica. The replication code can be found at \url{https://github.com/mcfraser3/generalizedRHP}. In this program the user can enter a $k$-tuple of forests from either $S_0$ or $S_3$ and the program will generate a list of graphs similar to the list we saw in Section 4, ending with the corresponding $k$-tuple of forests in $S_3$ or $S_0$ (respectively). Below is an example of the code output. The actual list that the code generates contains every edge moved in the order that they are moved, but for the sake of space we will only show the outcome of each application of $\phi_0,\phi_1,\phi_2$.

We begin with a set in $S_0$. Notice that the left-most graph is a tree and the three on the right are $5$-forests (there are no edges out of nodes $0,1,2,3,4$).
\begin{center}
    \includegraphics[scale=3.5]{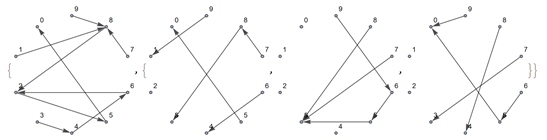}
\end{center}

$\phi_1$ moves edges from the tree to each of the forests.

\begin{center}
    \includegraphics[scale=3.5]{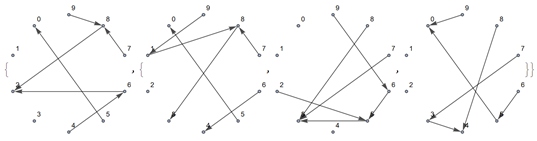}
\end{center}

The program highlights in red the meta-cycle that makes this a set of forbidden forests.

\begin{center}
    \includegraphics[scale=3.5]{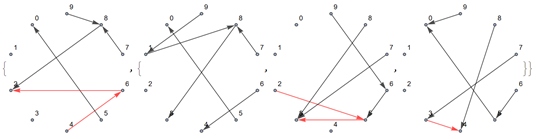}
\end{center}

The program applies $\phi_1$ again, this time with red edges.

\begin{center}
    \includegraphics[scale=3.5]{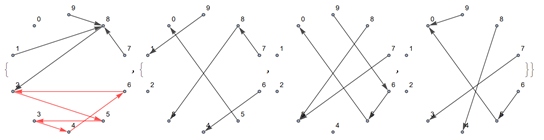}
\end{center}

Now the program changes the color of the cycle.

\begin{center}
    \includegraphics[scale=3.5]{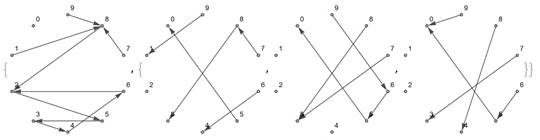}
\end{center}

Finally, after applying $\phi_1$, our program gives us the final set in $S_3$ that is the overall result of our bijection.

\begin{center}
    \includegraphics[scale=3.5]{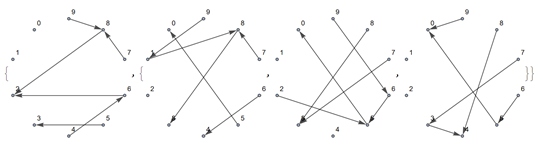}
\end{center}

\section{\large{Proofs of Generalized Forest and Dodgson/Muir Identities}}

We will restate and prove our two identities using the generalized Red Hot Potato algorithm.

\begin{theorem}
\textbf{Generalized Forest Identity.} For any $k$, let $\sigma_{\text{id}}$ be the identity permutation in $S_k$, and let $R^{NF}$ be the set $R^{\sigma_{\text{id}}}\setminus \bigcup_{\sigma\in S_k, \sigma\neq \sigma_{\text{id}}}R^{\sigma}$. Then
$$\sum_{(F_1,F_2,\dots,F_k)\in R_0\times (R_{[k]})^{k-1}}\prod_{i=1}^k a_{F_i}=\sum_{(F_1,F_2,\dots,F_k)\in R^{NF}}\prod_{i=1}^k a_{F_i}.$$
\end{theorem}

\begin{proof}
Because $\phi_0$, $\phi_1$, and $\phi_2$ are all sign-reversing involutions, by Theorem \ref{invothm}, the generalized Red Hot Potato algorithm is a bijection between $S_0$ and $S_3$. Notice that $S_0=(F_1,F_2,\dots,F_k)\in R_0\times (R_{[k]})^{k-1}$ and $S_3=(F_1,F_2,\dots,F_k)\in R^{NF}$. Thus we have a bijection between the sets each side of the identity sums over. Because the generalized Red Hot Potato algorithm merely moves edges between the forests in a given ordered set, $\prod_{i=1}^ka_{F_i}$ does not change after applying the algorithm. Thus the identity holds.
\end{proof}

\begin{theorem}
\textbf{Dodgson/Muir Identity.} Let $M$ be a square $n\times n$ matrix, and let $k$ be any integer $1\leq k\leq n$. Then
\begin{equation}
\begin{split}
\det(M)\cdot\det&(M[\{k+1,\dots,n\},\{k+1,\dots,n\}])^{k-1}=\\
&\sum_{\sigma\in S_k}\text{sgn}(\sigma)\prod_{i=1}^k\det(M[\{i,k+1,\dots,n\},\{\sigma(i),k+1,\dots,n\}]).
\end{split}
\end{equation}
\end{theorem}

\begin{proof}
As outlined in section 2, the Generalized Forest Identity proves the Dodgson/Muir Identity provided that $M=A_{0,0}$ where $A$ is the Laplacian for some directed graph. Let $M$ be an arbitrary $n\times n$ matrix. Construct a graph $G$ on node set $\{0,1,\dots,n\}$ such that every edge $i\to j$ has an edge weight given by negative the entry in the $i$th row, $j$th column of $M$ provided $i,j\neq 0$. Let edges $0\to i$ have weight $0$ for all $i$, and let edges $i\to 0$ have an edge weight given by the sum of the $i$th row of $M$ for all $i$. Then $M=A_{0,0}$ where $A$ is the Laplacian of $G$.
\end{proof}

\bibliographystyle{plain}

\end{document}